\documentclass[pdflatex,sn-mathphys-num]{sn-jnl}

\usepackage{graphicx}%
\usepackage{multirow}%
\usepackage{amsmath,amssymb,amsfonts}%
\usepackage{amsthm}%
\usepackage{mathrsfs}%
\usepackage[title]{appendix}%
\usepackage{xcolor}%
\usepackage{textcomp}%
\usepackage{manyfoot}%
\usepackage{booktabs}%
\usepackage{algorithm}%
\usepackage{algorithmicx}%
\usepackage{algpseudocode}%
\usepackage{listings}%

\theoremstyle{thmstyleone}%
\newtheorem{theorem}{Theorem}
\newtheorem{proposition}[theorem]{Proposition}%

\theoremstyle{thmstyletwo}%

\theoremstyle{thmstylethree}%
\newtheorem{definition}{Definition}%

\begin{document}

\title[2-Local and local derivations on Jordan rings]{2-Local and local derivations on Jordan matrix rings over commutative involutive rings}


\author[1]{\fnm{Shavkat} \sur{Ayupov}}\email{shavkat.ayupov@mathinst.uz}
\equalcont{These authors contributed equally to this work.}

\author*[1,2]{\fnm{Farkhodzhon} \sur{Arzikulov}}\email{arzikulovfn@gmail.com}
\equalcont{These authors contributed equally to this work.}

\author[2]{\fnm{Nodirbek} \sur{Umrzaqov}}\email{umrzaqov2010@mail.ru}
\equalcont{These authors contributed equally to this work.}

\author[2]{\fnm{Olimjon} \sur{Nuriddinov}}\email{o.nuriddinov86@mail.ru}
\equalcont{These authors contributed equally to this work.}

\affil*[1]{\orgdiv{V.I. Romanovskiy Institute of Mathematics}, \orgname{Uzbekistan Academy of Sciences}, \orgaddress{\street{Olmazor district, University street 9}, \city{Tashkent}, \postcode{100174}, \country{Uzbekistan}}}

\affil[2]{\orgdiv{Mathematics}, \orgname{Andijan State University}, \orgaddress{\street{University 129}, \city{Andijan}, \postcode{170100}, \country{Uzbekistan}}}




\abstract{In the present paper we prove that every 2-local inner
derivation on the Jordan ring of self-adjoint matrices over an involutive commutative associative ring
is a derivation. We also apply our technique
to various Jordan algebras of self-adjoint operator-valued maps on a set and prove
that every 2-local spatial derivation on such algebras is a spatial derivation.
It is also proved that every local spatial derivation on the same Jordan algebras is a derivation.}




\maketitle

\section{Introduction}

The present paper is devoted to 2-local and local derivations on Jordan rings and Jordan algebras.
The study of 2-local derivations began in the paper \cite{S} of \v{S}emrl. In \cite{S} the notion of 2-local derivations
is introduced and 2-local derivations on the algebra $B(H)$ of all bounded linear operators on an infinite-dimensional
separable Hilbert space $H$ are described. After these first results, a number of paper were devoted to 2-local maps on different types of rings,
algebras, Banach algebras and Banach spaces.

Let $\mathcal{A}$ be a unital involutive associative ring and $\mathcal{A}_ {sa}$ be the set of all self-adjoint elements of
the ring $\mathcal{A}$. Suppose $2$ is invertible in $\mathcal{A}$. Then, it is known that $(\mathcal{A} _{sa}, \cdot)$ is a Jordan ring
with respect to the Jordan multiplication $a\cdot b=\frac{1}{2}(ab+ba)$,
and, every inner derivation of the Jordan ring $(\mathcal{A}_{sa}, \cdot)$ is extended to
an inner derivation of the involutive associative ring $\mathcal{A}$. An extension of a derivation on
a special Jordan algebra to a derivation on the enveloping associative algebra of this Jordan algebra is considered in \cite{U}.
As to 2-local inner derivations on $(\mathcal{A}_{sa}, \cdot)$, in this case
till now, it was not possible to carry out an extension of a 2-local inner derivation on $(\mathcal{A}_{sa}, \cdot)$ to
a 2-local derivation on the associative ring $\mathcal{A}$ without additional conditions.
This problem shows that the problem of the description of 2-local derivations on Jordan algebras requires development of
methods for solving this problem within the framework of the theory of Jordan algebras.

In this paper 2-local derivations on the Jordan ring of self-adjoint matrices over an involutive commutative associative ring are described.
If in an associative algebra $\mathcal{A}$ we take the Jordan multiplication $a\cdot b=\frac{1}{2}(ab+ba)$, then we obtain
the Jordan algebra $(\mathcal{A},\cdot)$. In this case, every 2-local inner derivation of the algebra $\mathcal{A}$ is
a derivation if and only if every 2-local inner derivation of the Jordan algebra $(\mathcal{A},\cdot)$ is a derivation.
In general, for any special Jordan algebra $(\mathcal{J},\cdot)$, the Jordan multiplication $\cdot$ is generated by an
associative multiplication. Using this fact 2-local inner derivations on the Jordan
ring $H_n(\Re)$ of symmetric $n\times n$ matrices over a commutative associative ring $\Re$ are investigated in \cite{AA3}.
As a result, it was established that every 2-local inner derivation on the Jordan
ring $H_n(\Re)$ is a derivation. In particular, every 2-local inner derivation on the Jordan
rings $H_n({\mathbb{R}})$ and $S_n(\mathbb{C})$ of symmetric $n\times n$ matrices over the fields $\mathbb{R}$ of real numbers
and $\mathbb{C}$ of complex numbers respectively is a derivation.
Note that despite the connections $H_n(\mathbb{C})=H_n(\mathbb{R})+iSK_n(\mathbb{R})$ and
$S_n(\mathbb{C})=H_n(\mathbb{R})+iH_n(\mathbb{R})$,
where $SK_n(\mathbb{R})$ is the Lie algebra of skew-symmetric matrices on $\mathbb{R}$,
for the Jordan ring $H_n(\mathbb{C})$ of self-adjoint $n\times n$ matrices over the field $\mathbb{C}$ of complex numbers
the problem under consideration remained open.

Also, the method developed for the Jordan ring $H_n(\Re)$ in \cite{AA3} could not be applied to $H_n(\mathbb{C})$.
Indeed, if the involution of a unital involutive commutative associative ring $\Re$ is trivial, then
$$
\sum_{k=1}^m [a_k,b_k]^{i,i}=0, i=1,2,\dots,n,
$$
where $D=\sum_{k=1}^mD_{a_k,b_k}$ is an inner derivation on $H_n(\Re)$,
generated by $a_1$, $a_2$, $\dots,$ $a_m$, $b_1$, $b_2$, $\dots,$ $b_m\in H_n(\Re)$. In the general case
of unital involutive commutative associative rings this statement does not hold.
Also, first two equalities of \cite[Proposition 13]{AA3} and \cite[Proposition 14]{AA3} are also do not hold.
So, the method of proving of Theorem \ref{3.11} described in \cite{AA3} can not be applied to the case
of unital involutive commutative associative rings.

In the present paper the investigation of 2-local derivations on Jordan
algebras is also based on the properties of the associative multiplication.
We consider a similar problem in the more general
setting of Jordan rings.
Namely, we study inner derivations and 2-local inner
derivations on Jordan rings of self-adjoint matrices over an involutive commutative associative ring.
We prove that each 2-local inner derivation on the Jordan
ring $H_n(\Re)$ of self-adjoint $n\times n$ matrices over a unital involutive commutative associative ring $\Re$ is
a derivation.
As a corollary we establish that every 2-local inner derivation on the Jordan
algebra $H_n(\mathcal{A})$ of self-adjoint $n\times n$ matrices over a unital involutive commutative associative algebra $\mathcal{A}$
is a derivation. In one side, the method
employed is a generalization of the one in
\cite{AA3}, used in proving the corresponding theorem for the Jordan
ring $H_n(\Re)$ of symmetric $n\times n$ matrices over a unital commutative associative ring $\Re$.
In other side, the method employed is a Jordan modification of the one in \cite{AA3},
used in proving the corresponding theorem for the associative
ring $M_n(\Re)$ of $n\times n$ matrices over a unital commutative associative ring $\Re$. The general ideas of
these methods are the same but they distinct in details.

Indeed,

1) although the first two statements of Proposition \ref{3.31} and Proposition 1 in \cite{AA3} are the same, their proofs differ.

2) instead of $x_o=\sum^{n-1}_{k=1} e_{k,k+1}$ from \cite{AA3} we use $x_o=\sum_{k=1}^{n-1}\bar{e}_{k,k+1}\in H_n(\Re)$.

3) instead of $e_{i,j}$ from \cite{AA3} we use $\bar{e}_{i,j}=e_{i,j}+e_{j,i}$. So, the proofs concerning them differ in details.

4) Proposition \ref{3.41} is new, i.e., its analog does not exist in \cite{AA3}.

5) Proposition \ref{3.41} is used in proof of Proposition \ref{3.5}. So, although Proposition \ref{3.5} and proposition 2 in \cite{AA3} are the same, their proofs differ.

6) Definitions of $a_{i,j}$ in \cite{AA3} and  in Section \ref{S2} are differ. More precisely, $a_{i,j}$ in \cite{AA3} additionally defined
by the element $x_o=\sum^{n-1}_{k=1} e_{k,k+1}$.

We also study 2-local spatial derivations on various Jordan algebras of
self-adjoint operator-valued maps on a set.
We prove that every 2-local spatial derivation on subalgebras of the algebra
$M(\Omega,B(H)_{sa})$ of all maps from $\Omega$ to $B(H)_{sa}$ for an arbitrary set $\Omega$ and the algebra $B(H)$ of
all bounded linear operators on a separable Hilbert space $H$, under some conditions, is a derivation.
The problems considered here are firstly mentioned in \cite{AA4} (Problem 1).

The remaining section are devoted to the description of local inner derivations on the considered algebras.
It is proved that every local spatial derivation on various Jordan algebras of
self-adjoint operator-valued maps on a set is a derivation.
A number of results concerning local derivations on Jordan algebras are presented in
\cite{AKP}.

\section{Preliminaries}

This section is devoted to derivations and 2-local derivations of Jordan rings.

Let $\mathcal{J}$ be a Jordan ring. Recall that a map $D : \mathcal{J}\to \mathcal{J}$ is called
a derivation, if $D(x+y)=D(x)+D(y)$ and $D(xy)=D(x)y+xD(y)$ for
any two elements $x$, $y\in \mathcal{J}$.

Given subsets $B$ and $C$ of a Lie algebra with bracket $[\cdot, \cdot]$, let $[B,C]$ denote
the set of all finite sums of elements $[b,c]$, where $b\in B$ and $c\in C$.

Consider a Jordan ring $\mathcal{J}$ and let $m=\{xM: x\in \mathcal{J}\}$, where $xM$  denotes the multiplication
operator defined by $(xM)y:=x\cdot y$ for all $x$, $y\in \mathcal{J}$.
Let $aut(\mathcal{J})$ denotes the Lie ring of all derivations of $\mathcal{J}$.
The elements of the ideal $int(\mathcal{J}):=[m,m]$
of $aut(\mathcal{J})$ are called inner derivations of $\mathcal{J}$.
In other words, a derivation $D$ on $\mathcal{J}$ is called an inner derivation if there exist
elements  $a_1$, $a_2$, $\dots,$ $a_m$, $b_1$, $b_2$, $\dots,$ $b_m$ in  $\mathcal{J}$ such that
$$
D(x)=\sum_{k=1}^m [a_k(b_kx)-b_k(a_kx)], x\in \mathcal{J}.
$$
First the notion of an inner derivation introduced in \cite{NJ}.

A map $\Delta : \mathcal{J}\to \mathcal{J}$ is called a 2-local derivation, if for
any two elements $x$, $y\in \mathcal{J}$ there exists a derivation
$D_{x,y}:\mathcal{J}\to \mathcal{J}$ such that $\Delta (x)=D_{x,y}(x)$,
$\Delta(y)=D_{x,y}(y)$.

Let $\Delta$ be a 2-local derivation of the Jordan ring $\mathcal{J}$.
$\Delta$ is called a 2-local inner derivation, if for each pair of elements
$x$, $y\in \mathcal{J}$ there is an inner derivation $D_{x,y}$ of $\mathcal{J}$ such that $\Delta(x)=D_{x,y}(x)$, $\Delta(y)=D_{x,y}(y)$.

Let $\mathcal{A}$ be a unital associative ring where $2$ is invertible.
Then the set $\mathcal{A}$ with respect to the operations of
addition and Jordan multiplication
$$
a\cdot b=\frac{1}{2}(ab+ba), a, b\in \mathcal{A}
$$
is a Jordan ring. This Jordan ring will be denote by $(\mathcal{A}, \cdot)$. For any
elements $a_1$, $a_2$, $\dots,$ $a_m$, $b_1$, $b_2$, $\dots,$ $b_m\in \mathcal{A}$ the map
$$
D(x)=\sum_{k=1}^m D_{a_k,b_k}(x)=\sum_{k=1}^m(a_k\cdot (b_k\cdot x)-b_k\cdot (a_k\cdot x))
$$
$$
=D_{\frac{1}{4}(\sum_{k=1}^l [a_k,b_k])}(x), x\in \mathcal{A}
$$
is a derivation on the unital associative ring $\mathcal{A}$.
Therefore every inner derivation of the Jordan ring $(\mathcal{A}, \cdot)$ is an inner derivation of the
associative ring $\mathcal{A}$. And also every inner derivation of the form $D_{ab-ba}(x)=(ab-ba)x-x(ab-ba)$,
$x\in \mathcal{A}$ is an inner derivation of the Jordan ring $(\mathcal{A}, \cdot)$.
Also, every 2-local inner derivation of the Jordan ring $(\mathcal{A}, \cdot)$
is a 2-local inner derivation of the associative ring $\mathcal{A}$.

Let $\Re$ be a unital involutive commutative associative ring, $M_n(\Re)$ be the
$n\times n$ matrices ring over $\Re$, $n>1$.
Let $\{e_{i,j}\}_{i,j=1}^n$ be the set of matrix units in
$M_n(\Re)$, i.e. $e_{i,j}$ is a matrix with components
$a^{i,j}={\bf 1}$ and $a^{k,l}={\bf 0}$ if $(k,l)\neq (i,j)$, where
${\bf 1}$ is the identity element, ${\bf 0}$ is the zero element of
$\Re$, and a matrix $a\in M_n(\Re)$ is written as $a=\sum_{k,l=1}^n
a^{k,l}e_{k,l}$, where $a^{k,l}\in \Re$ for $k,l=1,2,\dots, n$.

Suppose that $2$ is invertible in $\Re$. In this case the set
$$
H_n(\Re)=\{(a^{i,j})_{i,j=1}^n\in M_n(\Re):
$$
$$
(a^{i,i})^*=a^{i,i},
(a^{i,j})^*=a^{j,i}, i\neq j,
i,j=1,2,\dots ,n\},
$$
i.e. the set of all self-adjoint $n\times n$ matrices over $\Re$,
is a Jordan ring with respect to the addition and the Jordan multiplication
$$
a\cdot b=\frac{1}{2}(ab+ba), a, b\in H_n(\Re).
$$
We denote this Jordan ring by $H_n(\Re)$.
Throughout this paper, let $\bar{e}_{i,j}=e_{i,j}+e_{j,i}$ for every pair of distinct indices $i$, $j$
in $\{1,2,\dots ,n\}$.

The following theorem takes place.

\begin{theorem} \cite[Theorem 15]{AA3} \label{3.111}
Let $\Re$ be a unital involutive commutative associative ring, and, let $H_n(\Re)$ be the Jordan ring of self-adjoint $n\times n$ matrices over $\Re$.
Then, if the involution is trivial, i.e., an identity map, then every 2-local inner derivation on $H_n(\Re)$ is a derivation.
\end{theorem}

\section{2-Local derivations on the Jordan ring of self-adjoint two dimensional matrices over an involutive commutative ring}

Here we prove the main theorem separately in the case of the Jordan ring of self-adjoint $2\times 2$ matrices over an involutive commutative ring.
In this case the theorem has an additional statement.

Let, throughout the rest sections, $\Re$ be a unital involutive commutative associative ring, and, $H_n(\Re)$ be the Jordan algebra of all self-adjoint $n\times n$ matrices over $\Re$. Let $D=\sum_{k=1}^m D_{a_k,b_k}$, be the inner derivations on $H_n(\Re)$,
generated by elements $a_1$, $a_2$, $\dots,$ $a_m$, $b_1$, $b_2$, $\dots,$ $b_m\in H_n(\Re)$.
Then, throughout the rest sections, the element $\frac{1}{4}(\sum_{k=1}^l [a_k,b_k])$, where $[a,b]=ab-ba$, we denote by $a_D$. Thus,
$a_D=\frac{1}{4}(\sum_{k=1}^l [a_k,b_k])$ and $D(x)=a_Dx-xa_D$, $x\in H_n(\Re)$.

\begin{proposition}  \label{3.1}
Let $D_1$, $D_2$ be inner derivations on $H_2(\Re)$ such that
$$
D_1(\bar{e}_{1,2})=D_2(\bar{e}_{1,2}).
$$
Then
$$
a_{D_1}^{1,2}=a_{D_2}^{1,2}, \,\,\,
a_{D_1}^{2,1}=a_{D_2}^{2,1}, \,\,\,
a_{D_1}^{1,1}-a_{D_1}^{2,2}=a_{D_2}^{1,1}-a_{D_2}^{2,2}
$$
with respect to the associative multiplication.
\end{proposition}

\begin{proof}
By the definition of a derivation and by the properties of Peirce component
$$
D_1(\bar{e}_{1,2})=D_1(2e_{11}\cdot \bar{e}_{1,2})=2(D_1(e_{11}))\cdot \bar{e}_{1,2}+2e_{11}\cdot D_1(\bar{e}_{1,2}),
$$
$$
D_2(\bar{e}_{1,2})=D_2(2e_{11}\cdot \bar{e}_{1,2})=2(D_2(e_{11}))\cdot \bar{e}_{1,2}+2e_{11}\cdot D_2(\bar{e}_{1,2}).
$$
Hence, since $D_1(\bar{e}_{1,2})=D_2(\bar{e}_{1,2})$ we have
$$
D_1(e_{11})\cdot \bar{e}_{1,2}=D_2(e_{11})\cdot \bar{e}_{1,2}.\,\,\,\,\,\,\,\,\,\,\,\,\,(1)
$$

Then, by (1), we have
\[
\bar{e}_{1,2}a_{D_1}e_{1,1}-e_{2,1}a_{D_1}+a_{D_1}e_{1,2}-e_{1,1}a_{D_1}\bar{e}_{1,2}=
\bar{e}_{1,2}a_{D_2}e_{1,1}-e_{2,1}a_{D_2}+a_{D_2}e_{1,2}-e_{1,1}a_{D_2}\bar{e}_{1,2}.
\]
Multiplying the last equality by $e_{1,1}$ on the left side yields
\[
e_{1,2}a_{D_1}e_{1,1}+e_{1,1}a_{D_1}e_{1,2}-e_{1,1}a_{D_1}\bar{e}_{1,2}=
e_{1,2}a_{D_2}e_{1,1}+e_{1,1}a_{D_2}e_{1,2}-e_{1,1}a_{D_2}\bar{e}_{1,2},
\]
and, multiplying the last equality by $e_{1,1}$ on the right side yields
\[
e_{1,2}a_{D_1}e_{1,1}-e_{1,1}a_{D_1}e_{2,1}=e_{1,2}a_{D_2}e_{1,1}-e_{1,1}a_{D_2}e_{2,1}.
\]
Multiplying this equality by $e_{1,2}$ on the right side we get
\[
e_{1,1}a_{D_1}e_{2,2}=e_{1,1}a_{D_2}e_{2,2}.\,\,\,\,\,\,\,\,\,\,\, (2)
\]
Similarly,
\[
e_{2,2}a_{D_1}e_{1,1}=e_{2,2}a_{D_2}e_{1,1}. \,\,\,\,\,\,\,\,\,\,\, (3)
\]
Then
\[
D_1(e_{1,1})=a_{D_1}e_{1,1}-e_{1,1}a_{D_1}=e_{2,2}a_{D_1}e_{1,1}+e_{1,1}a_{D_1}e_{1,1}-e_{1,1}a_{D_1}e_{1,1}+e_{1,1}a_{D_1}e_{2,2}
\]
\[
=e_{2,2}a_{D_2}e_{1,1}+e_{1,1}a_{D_2}e_{1,1}-e_{1,1}a_{D_2}e_{1,1}+e_{1,1}a_{D_2}e_{2,2}=a_{D_2}e_{1,1}-e_{1,1}a_{D_2}=D_2(e_{1,1})
\]
by (2) and (3), i.e.,
$$
D_1(e_{11})=D_2(e_{11}).\,\,\,\,\,\,\,\,\,\,\,\,\,(4)
$$

Take the extensions $D_{a_{D_1}}$, $D_{a_{D_2}}$ on $M_2(\Re)$ of the derivations $D_1$, $D_2$ respectively. Then
$$
D_{a_{D_1}}(e_{12})=D_{a_{D_1}}(e_{11}\bar{e}_{1,2})=D_{a_{D_1}}(e_{11})\bar{e}_{1,2}
+e_{11}D_{a_{D_1}}(\bar{e}_{1,2}),
$$
$$
D_{a_{D_2}}(e_{12})=D_{a_{D_2}}(e_{11}\bar{e}_{1,2})=D_{a_{D_2}}(e_{11})\bar{e}_{1,2}
+e_{11}D_{a_{D_2}}(\bar{e}_{1,2}).
$$
By (4) $D_{a_{D_1}}(e_{11})=D_{a_{D_2}}(e_{11})$. Therefore
$$
D_{a_{D_1}}(e_{12})=D_{a_{D_2}}(e_{12}).
$$
Hence,
$$
e_{11}a_{D_1}e_{12}e_{22}-e_{11}e_{12}a_{D_1}e_{22}=e_{11}a_{D_2}e_{12}e_{22}-e_{11}e_{12}a_{D_2}e_{22}
$$
and
$$
a_{D_1}^{11}e_{12}-a_{D_1}^{22}e_{12}=a_{D_2}^{11}e_{12}-a_{D_2}^{22}e_{12},
$$
i.e. $a_{D_1}^{1,2}=a_{D_2}^{1,2}$, $a_{D_1}^{2,1}=a_{D_2}^{2,1}$ and
$a_{D_1}^{11}-a_{D_1}^{22}=a_{D_2}^{11}-a_{D_2}^{22}$.
The proof is complete.
\end{proof}

The following theorem takes place.

\begin{theorem}  \label{3.3}
Let $\Re$ be a unital involutive commutative associative ring. Then
every inner 2-local derivation on $H_2(\Re)$ is an inner derivation.
\end{theorem}

\begin{proof}
Let $\Delta$ be an arbitrary 2-local inner derivation on $H_2(\Re)$.
We shall prove the existence of an inner derivation $D$ satisfying
$$
\Delta (x)=D_{a_D}(x)
$$
for all $x\in H_2(\Re)$.

Let $D_o$ be an inner derivation on $H_2(\Re)$ such that
$$
\Delta (\bar{e}_{1,2})=D(\bar{e}_{1,2}).
$$

Let
\[
a_{i,j}=e_{i,i}a_{D_o}e_{j,j}, a^{i,j}e_{i,j}=a_{D_o}^{i,j}e_{i,j},
\]
\[
a^{i,j}\in \Re, a_{D_o}^{i,j}\in \Re,  i,j=1,2.
\]
Then, by Proposition \ref{3.1}, for any inner derivation $D$ such that
$$
\Delta (\bar{e}_{1,2})=D(\bar{e}_{1,2}),
$$
we have
$$
a_{12}=e_{11}a_De_{22}, a_{21}=e_{22}a_De_{11},
$$
$$
a^{11}-a^{22}=a_D^{11}-a_D^{22}. \eqno{(5)}
$$
It should be noted that the elements $a_{1,2}$, $a_{2,1}\in H_2(\Re)$
and $a_{1,1}-a_{2,2}\in\Re$ do not depend on the inner derivations $D_o$ and $D$ chosen for the 2-local inner derivation $\Delta$.

Let $x\in H_2(\Re)$ and $D$ be an inner derivation on $H_2(\Re)$ such that
$$
\Delta (\bar{e}_{1,2})=D(\bar{e}_{1,2}), \Delta (x)=D(x).
$$
Then, by (5),
$$
e_{22}\Delta (x)e_{11}=e_{22}D(x)e_{11}=e_{22}D_{a_D}(x)e_{11}=e_{22}(a_Dx-xa_D)e_{11}
$$
$$
=e_{22}a_De_{22}xe_{11}+e_{22}a_De_{11}xe_{11}-e_{22}xe_{22}a_De_{11}-e_{22}xe_{11}a_De_{11}
$$
$$
=a_{21}e_{11}xe_{11}-e_{22}xe_{22}a_{21}+e_{22}a_De_{22}xe_{11}-e_{22}xe_{11}a_De_{11}
$$
$$
=a_{21}e_{11}xe_{11}-e_{22}xe_{22}a_{21}+(a^{22}-a^{11})e_{22}xe_{11}
$$
$$
=a_{21}e_{11}xe_{11}-e_{22}xe_{22}a_{21}+a_{22}e_{22}xe_{11}-e_{22}xe_{11}a_{11}
$$
$$
=e_{22}(a_{21}+a_{22}+a_{11}+a_{12})xe_{11}-e_{22}x(a_{21}+a_{11}+a_{11}+a_{12})e_{22}
$$
$$
=e_{22}((\sum_{i,j=1}^2 a_{i,j})x-x(\sum_{i,j=1}^2 a_{i,j}))e_{11}.
$$

Similarly
$$
e_{11}\Delta (x)e_{22}=e_{11}((\sum_{i,j=1}^2 a_{i,j})x-x(\sum_{i,j=1}^2 a_{i,j}))e_{22}.
$$
Also
$$
e_{11}\Delta (x)e_{11}=e_{11}D(x)e_{11}=e_{11}D_{a_D}(x)e_{11}=e_{11}(a_Dx-xa_D)e_{11}
$$
$$
=e_{11}a_De_{22}xe_{11}+e_{11}a_De_{11}xe_{11}-e_{11}xe_{22}a_De_{11}-e_{11}xe_{11}a_De_{11}
$$
$$
=a_{12}e_{22}xe_{11}+a_{11}e_{11}xe_{11}-e_{11}xe_{22}a_{21}-e_{11}xe_{11}a_{11}
$$
$$
=e_{11}(a_{12}+a_{11}+a_{22}+a_{21})xe_{11}+e_{11}x(a_{21}+a_{11}+a_{22}+a_{12})e_{11}
$$
$$
=e_{11}((\sum_{i,j=1}^2 a_{i,j})x-x(\sum_{i,j=1}^2 a_{i,j}))e_{11}.
$$
Similarly we have
$$
e_{22}\Delta (x)e_{22}=e_{22}((\sum_{i,j=1}^2 a_{i,j})x-x(\sum_{i,j=1}^2 a_{i,j}))e_{22}.
$$
Thus,
$$
\Delta (x)=\sum_{i,j=1}^2 e_{i,i}\Delta (x)e_{j,j}=\sum_{i,j=1}^2 e_{i,i}((\sum_{i,j=1}^2 a_{i,j})x-x(\sum_{i,j=1}^2 a_{i,j}))e_{j,j}
$$
$$
=(\sum_{i,j=1}^2 a_{i,j})x-x(\sum_{i,j=1}^2 a_{i,j})=D_o(x),
$$
where $a_{D_o}=\sum_{i,j=1}^2 a_{i,j}$.
So, $\Delta(x)=D_o(x)$ for all $x\in H_2(\Re)$. Hence $\Delta$ is an inner derivation. The proof
is complete.
\end{proof}

\section{2-Local derivations on the Jordan ring of self-adjoint matrices over an involutive commutative ring}  \label{S2}

In this section our main goal is to prove the following theorem.

\begin{theorem} \label{3.11}
Let $\Re$ be a unital involutive commutative associative ring, and, let $H_n(\Re)$ be the Jordan ring of self-adjoint $n\times n$ matrices over $\Re$.
Then every 2-local inner derivation on $H_n(\Re)$ is a derivation.
\end{theorem}

The idea of the proof is, for each 2-local inner derivation $\Delta$ on $H_n(\Re)$, we find a skew-adjoint element $\bar{a}$ in $M_n(\Re)$, depending on $\Delta$,
such that $\Delta(x)=\bar{a}x-x\bar{a}$ for any $x\in H_n(\Re)$ with respect to the multiplication of matrices.

The proof consists of three parts. The first part consists of constructing the element $\bar{a}$ (Proposition \ref{3.31}).
The second part consists of establishing properties of $\bar{a}$ needed for proving Theorem \ref{3.11} (Propositions \ref{3.6}, \ref{3.41}, \ref{3.5}, \ref{3.7} and \ref{3.8}).
The third part of the proof is a proof of the equality
\[
\Delta(x)=\bar{a}x-x\bar{a}, x\in H_n(\Re)
\]
(Proof of Theorem \ref{3.11}).

For this propose we need the following propositions.

\begin{proposition}  \label{3.31}
Let $i$, $j$ be arbitrary distinct indices in $\{1,2,\dots,n\}$, and,
let $D_1$, $D_2$ be the inner derivations on $H_n(\Re)$ such that
$$
D_1(\bar{e}_{i,j})=D_2(\bar{e}_{i,j}).
$$
Then the following equalities are valid relatively the associative multiplication
$$
a_{D_1}^{i,j}=a_{D_2}^{i,j}, \,\,\, a_{D_1}^{j,i}=a_{D_2}^{j,i}, \,\,\,
a_{D_1}^{i,i}-a_{D_1}^{j,j}=a_{D_2}^{i,i}-a_{D_2}^{j,j}.
$$
\end{proposition}

\begin{proof}
The assertion of this proposition is proved similar to the proof of Proposition \ref{3.1}.
\end{proof}

Let $\Delta$ be a 2-local derivation on $H_n(\Re)$, $i$, $j$ be arbitrary distinct indices in $\{1,2,\dots,n\}$, and,
let $D^{i,j}$ be the inner derivation on $H_n(\Re)$ such that
$$
\Delta (\bar{e}_{i,j})=D^{i,j}(\bar{e}_{i,j}).
$$
By Proposition \ref{3.31} the following elements are well-defined
$$
a_{i,j}=e_{i,i}a_{D^{i,j}}e_{j,j}, \,\,\,
a^{i,j}e_{i,j}=a_{D^{i,j}}^{i,j}e_{i,j}, \,\,\, a^{i,j}\in \Re, \,\,\,
a_{D^{i,j}}^{i,j}\in \Re,
$$
$$
a_{j,i}=e_{j,j}a_{D^{i,j}}e_{i,i}, \,\,\,
a^{j,i}e_{j,i}=a_{D^{i,j}}^{j,i}e_{j,i}, \,\,\,
a^{j,i}\in \Re, \,\,\, a_{D^{i,j}}^{j,i}\in \Re.
$$

Now let $x_o=\sum_{k=1}^{n-1}\bar{e}_{k,k+1}\in H_n(\Re)$.  Fix distinct indices $i_o$, $j_o$. Let
$D_o$ be an inner derivation on $H_n(\Re)$ such that
$$
\Delta(\bar{e}_{i_o,j_o})=D_o(\bar{e}_{i_o,j_o}), \,\,\, \Delta(x_o)=D_o(x_o).
$$

Put $a_{D_o}=\sum_{i,j=1}^n a_{D_o}^{i,j}e_{i,j}\in H_n(\Re)$, $a_{i,i}=a^{i,i}e_{i,i}=a_{D_o}^{i,i}e_{i,i}$, $i=1,2,\dots,n$ and
$\bar{a}=\sum_{i,j=1, i\neq j}^n a_{i,j}+\sum_{i=1}^n a_{i,i}$.
In these notations the following proposition is valid.

\begin{proposition}  \label{3.6}
Let $k$, $l$ be arbitrary distinct
indices, and let $D$ be an inner derivation on $H_n(\Re)$ such that
$$
\Delta(\bar{e}_{k,l})=D(\bar{e}_{k,l}), \,\,\,
\Delta(x_o)=D(x_o).
$$
Then $a^{k,k}-a^{l,l}=a_D^{k,k}-a_D^{l,l}$.
\end{proposition}

\begin{proof}
We may assume that  $k<l$. We have
$$
\Delta(x_o)=a_{D_o}x_o-x_oa_{D_o}=a_Dx_o-x_oa_D.
$$
Hence
$$
e_{k,k}(a_{D_o}x_o-x_oa_{D_o})e_{k+1,k+1}=e_{k,k}(a_Dx_o-x_oa_D)e_{k+1,k+1}
$$
and
$$
a_{D_o}^{k,k}-a_{D_o}^{k+1,k+1}=a_D^{k,k}-a_D^{k+1,k+1}.
$$
Then for the sequence
$$
(k,k+1),(k+1,k+2)\dots (l-1,l)
$$
we have
$$
a_{D_o}^{k,k}-a_{D_o}^{k+1,k+1}=a_D^{k,k}-a_D^{k+1,k+1},
a_{D_o}^{k+1,k+1}-a_{D_o}^{k+2,k+2}=a_D^{k+1,k+1}-a_D^{k+2,k+2},\dots
$$
$$
a_{D_o}^{l-1,l-1}-a_{D_o}^{l,l}=a_D^{l-1,l-1}-a_D^{l,l}.
$$
Hence
$$
a_{D_o}^{k,k}-a_D^{k,k}=a_{D_o}^{k+1,k+1}-a_D^{k+1,k+1},
a_{D_o}^{k+1,k+1}-a_D^{k+1,k+1}=a_{D_o}^{k+2,k+2}-a_D^{k+2,k+2},\dots
$$
$$
a_{D_o}^{l-1,l-1}-a_D^{l-1,l-1}=a_{D_o}^{l,l}-a_D^{l,l}.
$$
Therefore $a_{D_o}^{k,k}-a_D^{k,k}=a_{D_o}^{l,l}-a_D^{l,l}$, i.e.
$a_{D_o}^{k,k}-a_{D_o}^{l,l}=a^{k,k}-a^{l,l}=a_D^{k,k}-a_D^{l,l}$. The proof is complete.
\end{proof}

In these notations the following propositions are valid.

\begin{proposition} \label{3.41}
Let $\Delta$ be a 2-local derivation on $H_n(\Re)$,
$i$, $j$, $p$ be arbitrary pairwise distinct indices in $\{1,2,\dots n\}$, and,
let $D$ be the inner derivation on $H_n(\Re)$ such that
$$
\Delta (\bar{e}_{i,j})=D(\bar{e}_{i,j}).
$$
Then the following equalities hold relative to the associative multiplication
$$
e_{i,i}a_De_{p,p}=a_{i,p}, \,\,\,
e_{p,p}a_De_{j,j}=a_{p,j}.
$$
\end{proposition}

\begin{proof}
Let $\Delta$ be a 2-local derivation on $H_n(\Re)$,
let $D_1$, $D_2$ be the inner derivations on $H_n(\Re)$ such that
$$
\Delta (\bar{e}_{i,j})=D_1(\bar{e}_{i,j}), \Delta (\bar{e}_{i,j})=D_2(\bar{e}_{i,j}),
$$
$$
\Delta (\bar{e}_{i,p})=D_1(\bar{e}_{i,p}), \Delta (\bar{e}_{p,j})=D_2(\bar{e}_{p,j}).
$$
Then
$$
a_D\bar{e}_{i,j}-\bar{e}_{i,j}a_D=a_{D_1}\bar{e}_{i,j}-\bar{e}_{i,j}a_{D_1}=a_{D_2}\bar{e}_{i,j}-\bar{e}_{i,j}a_{D_2},
$$
$$
a_{i,p}=e_{i,i}a_{D_1}e_{p,p},  a_{p,j}=e_{p,p}a_{D_2}e_{j,j}.   \eqno{(*)}
$$
Also,
$$
a_D\bar{e}_{i,j}e_{p,p}-\bar{e}_{i,j}a_De_{p,p}=a_{D_1}\bar{e}_{i,j}e_{p,p}-\bar{e}_{i,j}a_{D_1}e_{p,p},
$$
$$
e_{p,p}a_D\bar{e}_{i,j}-e_{p,p}\bar{e}_{i,j}a_D=e_{p,p}a_{D_2}\bar{e}_{i,j}-e_{p,p}\bar{e}_{i,j}a_{D_2}
$$
and
$$
\bar{e}_{i,j}a_De_{p,p}=\bar{e}_{i,j}a_{D_1}e_{p,p}, e_{p,p}a_D\bar{e}_{i,j}=e_{p,p}a_{D_2}\bar{e}_{i,j}.
$$
Hence,
$$
e_{i,i}a_De_{p,p}=e_{i,i}a_{D_1}e_{p,p}, e_{p,p}a_De_{j,j}=e_{p,p}a_{D_2}e_{j,j}.
$$
By (*), we have
$$
e_{i,i}a_De_{p,p}=a_{i,p},  e_{p,p}a_De_{j,j}=a_{p,j}.
$$
These end the proof.
\end{proof}

\begin{proposition} \label{3.5}
Let $\Delta$ be a 2-local inner derivation on $H_n(\Re)$, $i$, $j$ be arbitrary distinct indices, and,
let $D$ be the inner derivation on $H_n(\Re)$ such that
$$
\Delta (\bar{e}_{i,j})=D(\bar{e}_{i,j}).
$$
Then the following equality is valid
$$
\Delta(\bar{e}_{i,j})=(\sum_{k,l=1, k\neq l}^n a_{k,l})\bar{e}_{i,j}-\bar{e}_{i,j}(\sum_{k,l=1, k\neq l}^n a_{k,l})
+(a_D^{i,i}-a_D^{j,j})e_{i,j}+(a_D^{j,j}-a_D^{i,i})e_{j,i}.
$$
\end{proposition}

\begin{proof}
By Proposition \ref{3.41}, we have
$$
\Delta(\bar{e}_{i,j})=a_D\bar{e}_{i,j}-\bar{e}_{i,j}a_D=\sum_{k=1}^ne_{k,k}
a_D\bar{e}_{i,j}-\sum_{k=1}^n\bar{e}_{i,j}a_De_{k,k}
$$
$$
=\sum_{k=1, k\neq i, k\neq j}^ne_{k,k} a_D\bar{e}_{i,j}-\sum_{k=1, k\neq i, k\neq
j}^n\bar{e}_{i,j}a_De_{k,k}
$$
$$
+(e_{i,i}+e_{j,j})a_D\bar{e}_{i,j}-\bar{e}_{i,j}a_D(e_{i,i}+e_{j,j})
$$
$$
=\sum_{k=1, k\neq i, k\neq j}^ne_{k,k} a_De_{i,j}-\sum_{k=1, k\neq i, k\neq j}^n e_{i,j}a_De_{k,k}
$$
$$
+\sum_{k=1, k\neq i, k\neq j}^ne_{k,k} a_D e_{j,i}-\sum_{k=1, k\neq i, k\neq j}^n e_{j,i}a_De_{k,k}
$$
$$
+(e_{i,i}+e_{j,j})a_D\bar{e}_{i,j}-\bar{e}_{i,j}a_D(e_{i,i}+e_{j,j})
$$
$$
=\sum_{k=1, k\neq i, k\neq j}^n a_{k,i}e_{i,j}-\sum_{k=1, k\neq i, k\neq j}^n e_{i,j}a_{j,k}
$$
$$
+\sum_{k=1, k\neq i, k\neq j}^n a_{k,j}e_{j,i}-\sum_{k=1, k\neq i, k\neq j}^n e_{j,i}a_{i,k}
$$
$$
+(e_{i,i}+e_{j,j})a_D\bar{e}_{i,j}-\bar{e}_{i,j}a_D(e_{i,i}+e_{j,j}).
$$
At the same time
$$
(e_{i,i}+e_{j,j})a_D\bar{e}_{i,j}-\bar{e}_{i,j}a_D(e_{i,i}+e_{j,j})
$$
$$
=e_{i,i}a_De_{i,j}+e_{i,i}a_De_{j,i}+e_{j,j}a_De_{i,j}+e_{j,j}a_De_{j,i}
$$
$$
-e_{i,j}a_De_{i,i}-e_{j,i}a_De_{i,i}-e_{i,j}a_De_{j,j}-e_{j,i}a_De_{j,j}
$$
$$
=(a_D^{i,i}-a_D^{j,j})e_{i,j}+(a_D^{j,j}-a_D^{i,i})e_{j,i}
$$
$$
+e_{i,i}a_De_{j,i}+e_{j,j}a_De_{i,j}-e_{i,j}a_De_{i,i}-e_{j,i}a_De_{j,j}
$$
$$
=(a_D^{i,i}-a_D^{j,j})e_{i,j}+(a_D^{j,j}-a_D^{i,i})e_{j,i}
$$
$$
+a_{i,j}e_{j,i}+a_{j,i}e_{i,j}-e_{i,j}a_{j,i}-e_{j,i}a_{i,j}.
$$
by Proposition \ref{3.31}. Hence,
$$
\Delta(\bar{e}_{i,j})
$$
$$
=(\sum_{k,l=1, k\neq l}^n a_{k,l})e_{i,j}-e_{i,j}(\sum_{k,l=1, k\neq l}^n a_{k,l})+
(\sum_{k,l=1, k\neq l}^n a_{k,l})e_{j,i}-e_{j,i}(\sum_{k,l=1, k\neq l}^n a_{k,l})
$$
$$
+(a_D^{i,i}-a_D^{j,j})e_{i,j}+(a_D^{j,j}-a_D^{i,i})e_{j,i}
$$
$$
=(\sum_{k,l=1, k\neq l}^n a_{k,l})\bar{e}_{i,j}-\bar{e}_{i,j}(\sum_{k,l=1, k\neq l}^n a_{k,l})
$$
$$
+(a_D^{i,i}-a_D^{j,j})e_{i,j}+(a_D^{j,j}-a_D^{i,i})e_{j,i}.
$$
This ends the proof.
\end{proof}

\begin{proposition}  \label{3.7}
Let $\Delta$ be an arbitrary inner 2-local derivation on $H_n(\Re)$.
Let $i$, $j$ be distinct indices in $\{1,2,\dots n\}$, and, let $D$ be an inner derivation such that
$$
\Delta (\bar{e}_{i,j})=D(\bar{e}_{i,j}).
$$
Then
$$
(1-e_{i,i})a_{D}e_{i,i}=(1-e_{i,i})\left(\sum_{k,q=1, k\neq q}^n a_{k,q}\right)e_{i,i},   \eqno{(6)}
$$
$$
e_{j,j}a_{D}(1-e_{j,j})=e_{j,j}\left(\sum_{k,q=1, k\neq q}^n a_{k,q})(1-e_{j,j}\right),   \eqno{(7)}
$$
$$
e_{i,i}a_{D}(1-e_{i,i})=e_{i,i}\left(\sum_{k,q=1, k\neq q}^n a_{k,q})(1-e_{i,i}\right),   \eqno{(8)}
$$
$$
(1-e_{j,j})a_{D}e_{j,j}=(1-e_{j,j})\left(\sum_{k,q=1, k\neq q}^n a_{k,q}\right)e_{j,j},
$$
$$
a^{i,i}-a^{j,j}=a_{D}^{i,i}-a_{D}^{j,j},
a^{j,j}-a^{i,i}=a_{D}^{j,j}-a_{D}^{i,i}.   \eqno{(9)}
$$
\end{proposition}

\begin{proof}
By the condition of the proposition, we have
$$
\Delta (\bar{e}_{i,j})=a_{D}\bar{e}_{i,j}-\bar{e}_{i,j}a_{D}, \Delta (x)=a_{D}x-xa_{D}.
$$
By Proposition \ref{3.5} we also have the following equalities
$$
\Delta(\bar{e}_{i,j})=a_D(\bar{e}_{i,j})-(\bar{e}_{i,j})a_D
$$
$$
=(e_{i,i}+e_{j,j})a_D(\bar{e}_{i,j})-(\bar{e}_{i,j})a_D(e_{i,i}+e_{j,j})
$$
$$
+(1-(e_{i,i}+e_{j,j}))a_D(\bar{e}_{i,j})-(\bar{e}_{i,j})a_D(1-(e_{i,i}+e_{j,j}))
$$
$$
=\left(\sum_{k,q=1, k\neq q}^n a_{k,q}\right)\bar{e}_{i,j}-\bar{e}_{i,j}\left(\sum_{k,q=1, k\neq q}^n a_{k,q}\right)
$$
$$
+\left(a_D^{i,i}-a_D^{j,j}\right)e_{i,j}+\left(a_D^{j,j}-a_D^{i,i}\right)e_{j,i}.
$$
Multiplying the last equality by $e_{i,i}+e_{j,j}$ on both, the left and right sides yields
$$
(e_{i,i}+e_{j,j})a_D(\bar{e}_{i,j})-(\bar{e}_{i,j})a_D(e_{i,i}+e_{j,j})
$$
$$
=(a_{i,j}+a_{j,i})\bar{e}_{i,j}-\bar{e}_{i,j}(a_{i,j}+a_{j,i})
$$
$$
+\left(a_D^{i,i}-a_D^{j,j}\right)e_{i,j}+\left(a_D^{j,j}-a_D^{i,i}\right)e_{j,i}.
$$
Hence,
$$
(1-(e_{i,i}+e_{j,j}))a_D\bar{e}_{i,j}+\bar{e}_{i,j}a_D(1-(e_{i,i}+e_{j,j}))
$$
$$
=(1-(e_{i,i}+e_{j,j}))\left(\sum_{k,q=1, k\neq q}^n a_{k,q}\right)\bar{e}_{i,j}+\bar{e}_{i,j}\left(\sum_{k,q=1, k\neq q}^n a_{k,q}\right)(1-(e_{i,i}+e_{j,j})).
$$
Multiplying the last equality by $e_{i,i}+e_{j,j}$ on the right side yields
$$
(1-(e_{i,i}+e_{j,j}))a_D\bar{e}_{i,j}=(1-(e_{i,i}+e_{j,j}))\left(\sum_{k,q=1, k\neq q}^n a_{k,q}\right)\bar{e}_{i,j}.
$$
Hence,
$$
\bar{e}_{i,j}a_D(1-(e_{i,i}+e_{j,j}))=\bar{e}_{i,j}\left(\sum_{k,q=1, k\neq q}^n a_{k,q}\right)(1-(e_{i,i}+e_{j,j})).
$$
From these it follows that
$$
(1-(e_{i,i}+e_{j,j}))a_De_{i,i}=(1-(e_{i,i}+e_{j,j}))\left(\sum_{k,q=1, k\neq q}^n a_{k,q}\right)e_{i,i},
$$
$$
e_{j,j}a_D(1-(e_{i,i}+e_{j,j}))=e_{j,j}\left(\sum_{k,q=1, k\neq q}^n a_{k,q}\right)(1-(e_{i,i}+e_{j,j})).
$$
At the same time
$$
e_{j,j}a_De_{i,i}=a_{j,i}
$$
by Proposition \ref{3.31}, and, hence,
$$
(1-e_{i,i})a_De_{i,i}=(1-e_{i,i})\left(\sum_{k,q=1, k\neq q}^n a_{k,q}\right)e_{i,i},
$$
$$
e_{j,j}a_D(1-e_{j,j})=e_{j,j}\left(\sum_{k,q=1, k\neq q}^n a_{k,q})(1-e_{j,j}\right).
$$
Similarly we have
$$
e_{i,i}a_D(1-e_{i,i})=e_{i,i}\left(\sum_{k,q=1, k\neq q}^n a_{k,q})(1-e_{i,i}\right),
$$
$$
(1-e_{j,j})a_De_{j,j}=(1-e_{j,j})\left(\sum_{k,q=1, k\neq q}^n a_{k,q}\right)e_{j,j}.
$$

Let $D_o$ be an inner derivation such that
$$
\Delta(\bar{e}_{i,j})=a_{D_o}\bar{e}_{i,j}-\bar{e}_{i,j}a_{D_o},
\Delta(x_o)=a_{D_o}x_o-x_oa_{D_o}.
$$

Let $a_{D_o}=\sum_{k,q=1}^n a_{D_o}^{k,q}e_{k,q}$. Then $a_{D_o}^{i,i}-a_{D_o}^{j,j}=a^{i,i}-a^{j,j}$ by Proposition \ref{3.6}. We have
$a_{D_o}^{i,i}-a_{D_o}^{j,j}=a_D^{i,i}-a_D^{j,j}$ since
$$
a_{D_o}\bar{e}_{i,j}-\bar{e}_{i,j}a_{D_o}=a_D\bar{e}_{i,j}-\bar{e}_{i,j}a_D.
$$
Hence
$$
a^{i,i}-a^{j,j}=a_D^{i,i}-a_D^{j,j},
a^{j,j}-a^{i,i}=a_D^{j,j}-a_D^{i,i}.
$$
The proof is complete.
\end{proof}

\begin{proposition}  \label{3.8}
Let $\Delta$ be an arbitrary inner 2-local derivation on $H_n(\Re)$.
Let $i$, $j$ be distinct indices in $\{1,2,\dots n\}$, and, let $D$, $\bar{D}$ be inner derivations such that
$$
\Delta (\bar{e}_{i,j})=D(\bar{e}_{i,j}), \,\, \Delta (e_{i,i})=\bar{D}(e_{i,i}).
$$
Then
$$
(1-e_{i,i})a_{\bar{D}}e_{i,i}=(1-e_{i,i})a_De_{i,i}, \,\,\,  e_{i,i}a_{\bar{D}}(1-e_{i,i})=e_{i,i}a_D(1-e_{i,i}),   \eqno{(10)}
$$
\end{proposition}

\begin{proof}
Let $v\in M_n(\Re)$ be an element such that
$$
\Delta(e_{i,i})=ve_{i,i}-e_{i,i}v,
\Delta(\bar{e}_{i,j})=v\bar{e}_{i,j}-\bar{e}_{i,j}v.
$$
Then
$$
(1-e_{i,i})ve_{i,i}-(1-e_{i,i})e_{i,i}v=(1-e_{i,i})a_{\bar{D}}e_{i,i}-(1-e_{i,i})e_{i,i}a_{\bar{D}},
$$
$$
ve_{i,i}(1-e_{i,i})-e_{i,i}v(1-e_{i,i})=a_{\bar{D}}e_{i,i}(1-e_{i,i})-e_{i,i}a_{\bar{D}}(1-e_{i,i}),
$$
$$
(1-(e_{i,i}+e_{j,j}))v\bar{e}_{i,j}-(1-(e_{i,i}+e_{j,j}))\bar{e}_{i,j}v
$$
$$
=(1-(e_{i,i}+e_{j,j}))a_D\bar{e}_{i,j}-(1-(e_{i,i}+e_{j,j}))\bar{e}_{i,j}a_D,
$$
$$
v\bar{e}_{i,j}(1-(e_{i,i}+e_{j,j}))-\bar{e}_{i,j}v(1-(e_{i,i}+e_{j,j}))
$$
$$
=a_D\bar{e}_{i,j}(1-(e_{i,i}+e_{j,j}))-\bar{e}_{i,j}a_D(1-(e_{i,i}+e_{j,j})).
$$
From this it follows that
$$
(1-e_{i,i})ve_{i,i}=(1-e_{i,i})a_{\bar{D}}e_{i,i}, e_{i,i}v(1-e_{i,i})=e_{i,i}a_{\bar{D}}(1-e_{i,i}),
$$
$$
(1-(e_{i,i}+e_{j,j}))ve_{i,j}=(1-(e_{i,i}+e_{j,j}))a_De_{i,j},
$$
$$
e_{j,i}v(1-(e_{i,i}+e_{j,j}))=e_{j,i}a_D(1-(e_{i,i}+e_{j,j})).
$$
By Proposition \ref{3.31}, $e_{j,j}ve_{i,i}=e_{j,j}a_De_{i,i}$, $e_{i,i}ve_{j,j}=e_{i,i}a_De_{j,j}$.
Hence,
$$
(1-(e_{i,i}+e_{j,j}))ve_{i,i}+e_{j,j}ve_{i,i}=(1-(e_{i,i}+e_{j,j}))a_De_{i,i}+e_{j,j}a_De_{i,i},
$$
$$
e_{i,i}v(1-(e_{i,i}+e_{j,j}))+e_{i,i}ve_{j,j}=e_{i,i}a_D(1-(e_{i,i}+e_{j,j}))+e_{i,i}a_De_{j,j}
$$
and
$$
(1-e_{i,i})ve_{i,i}=(1-e_{i,i})a_De_{i,i}, e_{i,i}v(1-e_{i,i})=e_{i,i}a_D(1-e_{i,i}).
$$
Therefore
$$
(1-e_{i,i})a_De_{i,i}=(1-e_{i,i})ve_{i,i}=(1-e_{i,i})a_{\bar{D}}e_{i,i},
$$
and
$$
e_{i,i}a_D(1-e_{i,i})=e_{i,i}v(1-e_{i,i})=e_{i,i}a_{\bar{D}}(1-e_{i,i}),
$$
i.e.,
$$
(1-e_{i,i})a_{\bar{D}}e_{i,i}=(1-e_{i,i})a_De_{i,i}, \,\,\,  e_{i,i}a_{\bar{D}}(1-e_{i,i})=e_{i,i}a_D(1-e_{i,i}).
$$
This ends the proof.
\end{proof}

\medskip

{\it Proof of Theorem \ref{3.11}.}
Let $\Delta$ be an arbitrary inner 2-local derivation on $H_n(\Re)$.
We will prove that $\Delta$ is a derivation on $H_n(\Re)$.

Let $x$ be an arbitrary element in $H_n(\Re)$, $i$, $j$ be distinct indices in $\{1,2,\dots n\}$, and, let
$D^{i,j}$ be an inner derivation such that
$$
\Delta (\bar{e}_{i,j})=D^{i,j}(\bar{e}_{i,j}), \Delta (x)=D^{i,j}(x)
$$
and $d(ij)=a_{D^{i,j}}$. Then, by (6), (7) and (9) of Proposition \ref{3.7}, we have
$$
e_{j,j}\Delta(x)e_{i,i}=e_{j,j}(d(ij)x-xd(ij))e_{i,i}
$$
$$
=e_{j,j}d(ij)(1-e_{j,j})xe_{i,i}+
e_{j,j}d(ij)e_{j,j}xe_{i,i}-e_{j,j}x(1-e_{i,i})d(ij)e_{i,i}-e_{j,j}xe_{i,i}d(ij)e_{i,i}
$$
$$
=e_{j,j}\left( \sum_{\xi,\eta=1, \xi\neq\eta}^n a^{\xi,\eta}e_{\xi,\eta}\right) xe_{i,i}
-e_{j,j}x \left(\sum_{\xi,\eta=1, \xi\neq\eta}^n a^{\xi,\eta}e_{\xi,\eta}\right)e_{i,i}
$$
$$
+ e_{j,j}d(ij)e_{j,j}xe_{i,i}-e_{j,j}xe_{i,i}d(ij)e_{i,i}
$$
$$
=e_{j,j}\left( \sum_{\xi,\eta=1,\xi\neq\eta}^n a^{\xi,\eta}e_{\xi,\eta}\right)xe_{i,i}
-e_{j,j}x\left(\sum_{\xi,\eta=1,\xi\neq\eta}^n a^{\xi,\eta}e_{\xi,\eta}\right)e_{i,i}
$$
$$
+e_{j,j}a^{j,j}e_{j,j}xe_{i,i}-e_{j,j}xe_{i,i}a^{i,i}e_{i,i}
$$
$$
=e_{j,j}\left(\sum_{\xi,\eta=1,\xi\neq\eta}^n a^{\xi,\eta}e_{\xi,\eta}\right)xe_{i,i}
-e_{j,j}x\left(\sum_{\xi,\eta=1,\xi\neq\eta}^n a^{\xi,\eta}e_{\xi,\eta}\right)e_{i,i}
$$
$$
+e_{j,j}\left(\sum_{\xi=1}^n a^{\xi,\xi}e_{\xi,\xi})\right)xe_{i,i}-e_{j,j}x\left(\sum_{\xi=1}^n a^{\xi,\xi}e_{\xi,\xi}\right)e_{i,i}
$$
$$
=e_{j,j}\left(\sum_{\xi,\eta=1}^n a^{\xi,\eta}e_{\xi,\eta}\right)xe_{i,i}-e_{j,j}x\left(\sum_{\xi,\eta=1}^n a^{\xi,\eta}e_{\xi,\eta}\right)e_{i,i}=
e_{j,j}(\bar{a}x-x\bar{a})e_{i,i}.
$$
So,
\[
e_{j,j}\Delta(x)e_{i,i}=e_{j,j}(\bar{a}x-x\bar{a})e_{i,i}, \,\,\, i,j=1,2,\dots ,n.     \eqno{(11)}
\]

\medskip

Let $d(ii)\in M_n(\Re)$ be an element such that
$$
\Delta(e_{i,i})=d(ii)e_{i,i}-e_{i,i}d(ii) \,\, \text{and}\,\,
\Delta(x)=d(ii)x-xd(ii),
$$
Then, by (6), (8) of Proposition \ref{3.7} and (10) of Proposition \ref{3.8} we have
$$
e_{i,i}\Delta(x)e_{i,i}=e_{i,i}(d(ii)x-xd(ii))e_{i,i}
$$
$$
=e_{i,i}d(ii)(1-e_{i,i})xe_{i,i}+
e_{i,i}d(ii)e_{i,i}xe_{i,i}-e_{i,i}x(1-e_{i,i})d(ii)e_{i,i}-e_{i,i}xe_{i,i}d(ii)e_{i,i}
$$
$$
=e_{i,i}d(ij)(1-e_{i,i})xe_{i,i}+
e_{i,i}d(ii)e_{i,i}xe_{i,i}-e_{i,i}x(1-e_{i,i})d(ij)e_{i,i}-e_{i,i}xe_{i,i}d(ii)e_{i,i}
$$
$$
=e_{i,i}\left(\sum_{\xi,\eta=1,\xi\neq\eta}^n a^{\xi,\eta}e_{\xi,\eta}\right)xe_{i,i}
-e_{i,i}x\left(\sum_{\xi,\eta=1,\xi\neq\eta}^n a^{\xi,\eta}e_{\xi,\eta}\right)e_{i,i}
$$
$$
+e_{i,i}d(ii)e_{i,i}xe_{i,i}-e_{i,i}xe_{i,i}d(ii)e_{i,i}
=e_{i,i}\left(\sum_{\xi,\eta=1,\xi\neq\eta}^n a^{\xi,\eta}e_{\xi,\eta}\right)xe_{i,i}
$$
$$
-e_{i,i}x\left(\sum_{\xi,\eta=1,\xi\neq\eta}^n a^{\xi,\eta}e_{\xi,\eta}\right)e_{i,i}
+a^{i,i}e_{i,i}xe_{i,i}-e_{i,i}xa^{i,i}e_{i,i}
$$
$$
=e_{i,i}\left(\sum_{\xi,\eta=1,\xi\neq\eta}^n a^{\xi,\eta}e_{\xi,\eta}\right)xe_{i,i}
-e_{i,i}x\left(\sum_{\xi,\eta=1,\xi\neq\eta}^n a^{\xi,\eta}e_{\xi,\eta}\right)e_{i,i}
$$
$$
+e_{i,i}\left(\sum_{\xi=1}^n a^{\xi,\xi}e_{\xi,\xi}\right)xe_{i,i}
-e_{i,i}x\left(\sum_{\xi=1}^n a^{\xi,\xi}e_{\xi,\xi}\right)e_{i,i}
$$
$$
=e_{i,i}\left(\sum_{\xi,\eta=1}^n a^{\xi,\eta}e_{\xi,\eta}\right)xe_{i,i}
-e_{i,i}x\left(\sum_{\xi,\eta=1}^n a^{\xi,\eta}e_{\xi,\eta}\right)e_{i,i}=e_{i,i}(\bar{a}x-x\bar{a})e_{i,i}.
$$
Hence,
$$
e_{i,i}\Delta(x)e_{i,i}=e_{i,i}(\bar{a}x-x\bar{a})e_{i,i}, \,\,\, i=1,2,\dots ,n.   \eqno{(12)}
$$

Thus, by (11) and (12) we have
$$
\Delta(x)=\sum_{i,j=1}^ne_{i,i}\Delta(x)e_{j,j}=\sum_{i,j=1}^ne_{i,i}(\bar{a}x-x\bar{a})e_{j,j}=\bar{a}x-x\bar{a}
$$
for each $x\in H_n(\Re)$. Note that, by the construction, $\bar{a}$ is a self-adjoint element, i.e. $\bar{a}\in H_n(\Re)$,
and $\bar{a}x-x\bar{a}\in H_n(\Re)$ for each $x\in H_n(\Re)$.
It can be straightforwardly checked that $\Delta$ is a derivation on $H_n(\Re)$.
The proof is complete.
\hfill $\Box$

The proof of Theorem \ref{3.11} is also valid for Jordan algebras of self-adjoint matrices over
a commutative involutive associative algebra. Therefore, if we define a derivation, an inner derivation, a 2-local derivation, a 2-local inner derivation
on Jordan algebras similarly to the case of Jordan rings, then the following theorem
takes place.

\begin{theorem} \label{2.7}
Let $\mathcal{A}$ be a unital commutative involutive associative algebra, and let $H_n(\mathcal{A})$
be the Jordan algebra of all self-adjoint $n\times n$ matrices over $\mathcal{A}$. Then
any 2-local inner derivation on the Jordan algebra $H_n(\mathcal{A})$ is a derivation.
\end{theorem}

\section{2-local derivations on Jordan algebras of self-adjoint operator-valued
maps}

Throughout the present section, let $\aleph_o$ be the countable cardinal
number,
$H$ be a separable Hilbert space of dimension $\aleph_o$ over ${\mathbb F}={\mathbb R}$, ${\mathbb C}$, and, let $B(H)$ be the algebra of
all bounded linear operators on $H$. Let $\{e_i\}_{i=1}^\infty$ be a maximal family of orthogonal minimal
projections in $B(H)$ and, let $\{e_{i,j}\}_{i,j=1}^\infty$ be the family of matrix units defined
by $\{e_i\}_{i=1}^\infty$, i.e. $e_{i,i}=e_i$, $e_{i,i}=e_{i,j}e_{j,i}$ and $e_{j,j}=e_{j,i}e_{i,j}$ for
each pair of natural numbers $i$, $j$.

Let $B_{sa}(H)$ be the vector space of all self-adjoint operators in $B(H)$, i.e.
$$
B_{sa}(H)=\{a\in B(H): a^*=a\}.
$$
Then with respect to Jordan multiplication
$$
a\cdot b=\frac{1}{2}(ab+ba), a,b\in B_{sa}(H)
$$
$B_{sa}(H)$ is a Jordan algebra.

Let $\Omega$ be an arbitrary set, $M(\Omega,B_{sa}(H))$ be the
Jordan algebra of all maps from $\Omega$ to $B_{sa}(H)$. Put
$$
\hat{e}_{i,j}={\bf 1}e_{i,j},
$$
where ${\bf 1}$ is the unit of the algebra $M(\Omega)$ of all ${\mathbb F}$-valued maps on $\Omega$.
Throughout the rest sections, put $s_{i,j}=\hat{e}_{i,j}+\hat{e}_{j,i}$ for all
distinct natural numbers $i$, $j$.

\medskip

\begin{definition}
Let $A$ be a Jordan algebra and, let $B$ be a Jordan subalgebra of $A$.
A derivation $D$ on $B$ is said to be spatial, if $D$ is implemented by elements from $A$,
i.e.,
$$
D(x)=\sum_{k=1}^m D_{a_k,b_k}(x)=\sum_{k=1}^m(a_k\cdot (b_k\cdot x)-b_k\cdot (a_k\cdot x)), x\in B,
$$
for some $a_1$, $a_2$, $\dots,$ $a_m$, $b_1$, $b_2$, $\dots,$ $b_m\in A$.
A 2-local derivation $\Delta$ on $B$ is called 2-local spatial derivation
implemented by elements from $A$, if for every two elements $x$, $y\in B$ there
exists a spatial derivation $D$
on $B$ such that $\Delta(x)=D(x)$, $\Delta(y)=D(y)$.
\end{definition}

Let, throughout the present section, $\Omega$ be an arbitrary set, $M(\Omega,B_{sa}(H))$ be the
Jordan algebra of all maps from $\Omega$ to $B_{sa}(H)$.
Let $(\lambda_n)$ be a sequence of nonzero numbers from ${\mathbb F}$ such that $\sum_{n=1}^\infty \lambda_n\lambda_n^*<\infty$ and
$x_o=\sum_{n=1}^\infty\lambda_n s_{n,n+1}\in M(\Omega,B_{sa}(H))$.
Let
$\mathcal{J}$ be a Jordan subalgebra of $M(\Omega,B_{sa}(H))$ containing the element $x_o$ and the family
$\{\hat{e}_{i,i}\}_{i=1}^\infty\cup\{s_{i,j}\}_{i,j=1,i\neq j}^\infty$, and $\Delta$ be a 2-local spatial derivation
on $\mathcal{J}$ implemented by elements from $M(\Omega,B_{sa}(H))$. Then the following propositions hold.

\begin{proposition}  \label{3.4111}
Let $\Delta$ be a 2-local derivation on $M(\Omega,B_{sa}(H))$, and,
let $i_1$, $i_2$, $\dots$, $i_n$ be arbitrary pairwise distinct natural numbers, $e=\hat{e}_{i_1,i_1}+\hat{e}_{i_2,i_2}+\dots+\hat{e}_{i_n,i_n}$.
Then the mapping
$$
\Delta_{e}(x)=e\Delta(x)e, x\in eM(\Omega,B_{sa}(H))e
$$
is a derivation on $eM(\Omega,B_{sa}(H))e$ and there exists an element $d\in eM(\Omega,B(H))e$
such that
$$
\Delta_{e}(x)=dx-xd, x\in eM(\Omega,B_{sa}(H))e.
$$
If $f=\hat{e}_{j_1,j_1}+\hat{e}_{j_2,j_2}+\dots+\hat{e}_{j_p,j_p}$ for indices $j_1$, $j_2$, $\dots$, $j_p$ from $i_1$, $i_2$, $\dots$, $i_n$,
then
$$
\hat{e}_{i,i}d\hat{e}_{j,j}=\hat{e}_{i,i}c\hat{e}_{j,j}, d^{i,i}-d^{j,j}=c^{i,i}-c^{j,j}
$$
for arbitrary distinct indices $i$, $j$ from $\{j_1, j_2, \dots, j_p\}$,
where $c$ is an element in $fM(\Omega,B(H))f$ such that
$$
\Delta_f(x)=f\Delta(x)f, x\in fM(\Omega,B_{sa}(H))f.
$$
\end{proposition}

\begin{proof}
Similar to proof of \cite[Proposition 2.7]{NP} it can be proved that
$\Delta_{e}$ is a 2-local derivation on $eM(\Omega,B_{sa}(H))e$. We have
$eM(\Omega,B(H))e\cong M_n({\mathbb F})\otimes M(\Omega)$. Therefore,
by Theorem \ref{3.11} and its proof $\Delta_{e}$ is a derivation and there exists
an element $d$ in $eM(\Omega,B(H))e$ such that
$$
\Delta_{e}(x)=dx-xd, x\in eM(\Omega,B_{sa}(H))e.
$$
Similarly, there exists
an element $c$ in $fM(\Omega,B(H))f$ such that
$$
\Delta_{f}(x)=dx-xd, x\in fM(\Omega,B_{sa}(H))f.
$$
For arbitrary distinct $i$ and $j$ from $\{j_1, j_2, \dots, j_p\}$ we have
$$
\Delta_{f}(\hat{e}_{i,i})=f\Delta_{e}(\hat{e}_{i,i})f,
$$
i.e.,
$$
c\hat{e}_{i,i}-\hat{e}_{i,i}c=fd\hat{e}_{i,i}-\hat{e}_{i,i}df.
$$
Hence,
\[
\hat{e}_{i,i}(c\hat{e}_{i,i}-\hat{e}_{i,i}c)\hat{e}_{j,j}=\hat{e}_{i,i}(fd\hat{e}_{i,i}-\hat{e}_{i,i}df)\hat{e}_{j,j}
\]
and
\[
\hat{e}_{i,i}c\hat{e}_{j,j}=\hat{e}_{i,i}d\hat{e}_{j,j}.
\]
We also have
$$
\Delta_{f}(s_{i,j})=f\Delta_{e}(s_{i,j})f,
$$
i.e.,
$$
cs_{i,j}-s_{i,j}c=fds_{i,j}-s_{i,j}df.
$$
Hence,
\[
\hat{e}_{i,i}(cs_{i,j}-s_{i,j}c)\hat{e}_{j,j}=\hat{e}_{i,i}(fds_{i,j}-s_{i,j}df)\hat{e}_{j,j}
\]
and
\[
\hat{e}_{i,i}c\hat{e}_{i,j}-\hat{e}_{i,j}c\hat{e}_{j,j}=\hat{e}_{i,i}d\hat{e}_{i,j}-\hat{e}_{i,j}d\hat{e}_{j,j},
\]
i.e.,
\[
c^{i,i}-c^{j,j}=d^{i,i}-d^{j,j}.
\]
The proof is complete.
\end{proof}

Let $D=\sum_{k=1}^m D_{a_k,b_k}$, be a derivations on $\mathcal{J}$,
generated by elements $a_1$, $a_2$, $\dots,$ $a_m$, $b_1$, $b_2$, $\dots,$ $b_m\in M(\Omega,B_{sa}(H))$.
Then, throughout the present section, the element $\frac{1}{4}(\sum_{k=1}^l [a_k,b_k])$ we denote by $a_D$,
i.e., $a_D=\frac{1}{4}(\sum_{k=1}^l [a_k,b_k])$.

Let $\Delta$ be a 2-local spatial derivation on $\mathcal{J}$, $i$, $j$ be arbitrary pairwise distinct natural numbers, and,
let $D=\sum_{k=1}^m D_{a_k,b_k}$ be a derivation on $\mathcal{J}$, generated by elements
$a_1$, $a_2$, $\dots,$ $a_m$, $b_1$, $b_2$, $\dots,$ $b_m\in M(\Omega,B_{sa}(H))$ such that
$$
\Delta (s_{i,j})=\sum_{k=1}^m D_{a_k,b_k}(s_{i,j}).
$$
By Proposition \ref{3.4111} the following elements are well-defined
$$
a_{i,j}=\hat{e}_{i,i}a_D\hat{e}_{j,j}, a^{i,j}\hat{e}_{i,j}=a_D\hat{e}_{i,j},
a^{i,j}=a_D^{i,j}\in M(\Omega),
$$
$$
a_{j,i}=\hat{e}_{j,j}a_D\hat{e}_{i,i}, a^{j,i}\hat{e}_{j,i}=a_D\hat{e}_{j,i},
a^{j,i}=a_D^{j,i}\in M(\Omega).
$$

Fix different indices $i_o$, $j_o$. Let
$x_1$, $x_2$, $\dots,$ $x_l$, $y_1$, $y_2$, $\dots,$ $y_l\in M(\Omega,B_{sa}(H))$ be elements such that
$$
\Delta(s_{i_oj_o})=\sum_{k=1}^l [x_k(y_ks_{i_oj_o})-y_k(x_ks_{i_oj_o})]
$$
and
$$
\Delta(x_o)=\sum_{k=1}^l [x_k(y_kx_o)-y_k(x_kx_o)].
$$
Let $c=a_D$, where $D=\sum_{k=1}^l D_{x_k,y_k}$.
Put $c=\sum_{i,j=1}^\infty c^{i,j}\hat{e}_{i,j}\in M(\Omega,B_{sa}(H))$ and
$\bar{a}=\sum_{i,j=1}^\infty a_{i,j}$, where $a_{i,i}=c^{i,i}\hat{e}_{i,i}$ for each $i$.
In the above notations we have the following propositions.

\medskip

\begin{proposition} \label{4.511}
Let $k$, $l$ be arbitrary different
natural numbers, and let $c_1$, $c_2$, $\dots,$ $c_m$, $d_1$, $d_2$, $\dots,$ $d_m$ be elements in $M(\Omega,B_{sa}(H))$ such that
$$
\Delta(s_{k,l})=\sum_{p=1}^m [c_p(d_ps_{k,l})-d_p(c_ps_{k,l})]
$$
and
$$
\Delta(x_o)=\sum_{p=1}^m [c_p(d_px_o)-d_p(c_px_o)].
$$
Then $c^{k,k}-c^{l,l}=b^{k,k}-b^{l,l}$, where $b=a_D$, $D=\sum_{k=1}^m D_{c_k,d_k}$.
\end{proposition}

\begin{proof}
We may assume that  $k<l$. We have
$$
\Delta(x_o)=cx_o-x_oc=bx_o-x_ob.
$$
Hence
$$
\hat{e}_{k,k}(cx_o-x_oc)\hat{e}_{k+1,k+1}=\hat{e}_{k,k}(bx_o-x_ob)\hat{e}_{k+1,k+1},
$$
$$
\hat{e}_{k,k}(c\lambda_k s_{k,k+1}-\lambda_k s_{k,k+1}c)\hat{e}_{k+1,k+1}=\hat{e}_{k,k}(b\lambda_ks_{k,k+1}-\lambda_ks_{k,k+1}b)\hat{e}_{k+1,k+1},
$$
$$
\lambda_k \hat{e}_{k,k}c\hat{e}_{k,k+1}-\lambda_k \hat{e}_{k,k+1}c\hat{e}_{k+1,k+1}=\lambda_k\hat{e}_{k,k}b\hat{e}_{k,k+1}-\lambda_k\hat{e}_{k,k+1}b\hat{e}_{k+1,k+1},
$$
and
$$
c^{k,k}-c^{k+1,k+1}=b^{k,k}-b^{k+1,k+1}.
$$
Thus, this proposition is proved similarly to the proof of
Proposition \ref{3.6}.
\end{proof}

\begin{proposition}  \label{4.2}
Let $\Delta$ be a 2-local spatial derivation on $\mathcal{J}$ implemented by elements in $M(\Omega,B_{sa}(H))$,
let $i$, $j$, $k$ be arbitrary pairwise distinct indices, and
$D=\sum_{k=1}^l D_{c_k,d_k}$ be the derivation on $\mathcal{J}$, generated by some elements
$c_1$, $c_2$, $\dots,$ $c_l$, $d_1$, $d_2$, $\dots,$ $d_l\in M(\Omega,B_{sa}(H))$ such that
$$
\Delta (s_{i,j})=\sum_{k=1}^l D_{c_k,d_k}(s_{i,j}).
$$
Then the following equalities are valid
$$
{a_D}^{i,j}=a^{i,j}, \,\,\, {a_D}^{j,i}=a^{j,i}, \,\,\,
{a_D}^{i,i}-{a_D}^{j,j}=a^{i,i}-a^{j,j}.
$$
\end{proposition}

\begin{proof}
The proof is similar to the proof of Proposition \ref{3.31}.
\end{proof}

\begin{proposition}  \label{4.21}
Let $\Delta$ be a 2-local spatial derivation on $\mathcal{J}$ implemented by elements in $M(\Omega,B_{sa}(H))$, and,
let $i$, $j$, $k$ be arbitrary pairwise distinct indices, and
$D=\sum_{k=1}^l D_{c_k,d_k}$ be the derivation on $\mathcal{J}$, generated by some elements
$c_1$, $c_2$, $\dots,$ $c_l$, $d_1$, $d_2$, $\dots,$ $d_l\in M(\Omega,B_{sa}(H))$ such that
$$
\Delta (s_{i,j})=\sum_{k=1}^l D_{c_k,d_k}(s_{i,j}).
$$
Then the following equalities are valid relative to the associative multiplication
$$
\hat{e}_{i,i}{a_D}\hat{e}_{k,k}=a_{i,k}, \,\,\, \hat{e}_{k,k}{a_D}\hat{e}_{j,j}=a_{k,j}.
$$
\end{proposition}

\begin{proof}
The proof is similar to the proof of Proposition \ref{3.41}.
\end{proof}

\begin{proposition} \label{4.5}
Let $\Delta$ be a 2-local spatial derivation on $\mathcal{J}$, and,
let $i$, $j$ be arbitrary different indices and let $D=\sum_{k=1}^l D_{c_k,d_k}$ be a derivation on $\mathcal{J}$,
generated by the elements $c_1$, $c_2$, $\dots,$ $c_l$, $d_1$, $d_2$, $\dots,$ $d_l\in M(\Omega,B_{sa}(H))$ such that
$$
\Delta (s_{i,j})=\sum_{k=1}^l D_{c_k,d_k}(s_{i,j}).
$$
Then the following equality is valid
$$
\Delta(s_{i,j})=\left(\sum_{k\neq l} a_{k,l}\right)s_{i,j}-s_{i,j}\left(\sum_{k\neq l} a_{k,l}\right)
$$
$$
+\left({a_D}^{i,i}-{a_D}^{j,j}\right)\hat{e}_{i,j}+\left({a_D}^{j,j}-{a_D}^{i,i}\right)\hat{e}_{j,i}.
$$
\end{proposition}

\begin{proof}
The present proposition is proved similarly to the proof of Proposition \ref{3.5}.
\end{proof}

\begin{proposition}  \label{4.61}
Let $\Delta$ be a 2-local spatial derivation on $\mathcal{J}$, and,
let $i$, $j$ be distinct indices, and, let $D$ be a spatial derivation on $\mathcal{J}$ implemented by elements in $M(\Omega,B_{sa}(H))$ such that
$$
\Delta (s_{i,j})=D(s_{i,j}).
$$
Then
$$
(1-\hat{e}_{i,i})a_{D}\hat{e}_{i,i}=(1-\hat{e}_{i,i})\left(\sum_{k,q=1, k\neq q}^n a_{k,q}\right)\hat{e}_{i,i},
$$
$$
\hat{e}_{j,j}a_{D}(1-\hat{e}_{j,j})=\hat{e}_{j,j}\left(\sum_{k,q=1, k\neq q}^n a_{k,q})(1-\hat{e}_{j,j}\right),
$$
$$
\hat{e}_{i,i}a_{D}(1-\hat{e}_{i,i})=\hat{e}_{i,i}\left(\sum_{k,q=1, k\neq q}^n a_{k,q})(1-\hat{e}_{i,i}\right),
$$
$$
(1-\hat{e}_{j,j})a_{D}\hat{e}_{j,j}=(1-\hat{e}_{j,j})\left(\sum_{k,q=1, k\neq q}^n a_{k,q}\right)\hat{e}_{j,j},
$$
$$
a^{i,i}-a^{j,j}=a_{D}^{i,i}-a_{D}^{j,j},
a^{j,j}-a^{i,i}=a_{D}^{j,j}-a_{D}^{i,i}.
$$
\end{proposition}

\begin{proof}
The present proposition is proved similarly to the proof of Proposition \ref{3.7}.
\end{proof}

\begin{proposition}  \label{4.8}
Let $\Delta$ be an arbitrary 2-local spatial derivation on $\mathcal{J}$.
Let $i$, $j$ be distinct indices, and, let $D$, $\bar{D}$ be spatial derivations on $\mathcal{J}$
implemented by elements in $M(\Omega,B_{sa}(H))$ such that
$$
\Delta (s_{i,j})=D(s_{i,j}), \,\, \Delta (\hat{e}_{i,i})=\bar{D}(\hat{e}_{i,i}).
$$
Then
$$
(1-\hat{e}_{i,i})a_{\bar{D}}\hat{e}_{i,i}=(1-\hat{e}_{i,i})a_D\hat{e}_{i,i}, \,\,\,
\hat{e}_{i,i}a_{\bar{D}}(1-\hat{e}_{i,i})=\hat{e}_{i,i}a_D(1-\hat{e}_{i,i}),   \eqno{(10)}
$$
\end{proposition}

\begin{proof}
The present proposition is proved similarly to the proof of Proposition \ref{3.8}.
\end{proof}

The following theorem is the key result of this section.

\begin{theorem} \label{4.4}
Let $\Omega$ be an arbitrary set, $M(\Omega,B_{sa}(H))$ be the
Jordan algebra of all maps from $\Omega$ to $B_{sa}(H)$. Let
$\mathcal{J}$ be a Jordan subalgebra of $M(\Omega,B_{sa}(H))$ containing the element $x_o$ and the family
$\{\hat{e}_{i,i}\}_{i=1}^\infty\cup\{s_{i,j}\}_{i,j=1,i\neq j}^\infty$. Then any 2-local spatial derivation
on $\mathcal{J}$ implemented by elements in $M(\Omega,B_{sa}(H))$
is a derivation.
\end{theorem}

\begin{proof}
We prove that each 2-local inner derivation $\Delta$ on $\mathcal{J}$ satisfies
the condition
$$
\Delta(x)=\bar{a}x-x\bar{a}, x\in \mathcal{J}
$$
for the element $\bar{a}\in M(\Omega,B_{sa}(H))$ defined above.
Let $i$, $j$ be arbitrary distinct natural numbers. The proofs of the equalities
$$
\hat{e}_{j,j}\Delta(x)\hat{e}_{i,i}=\hat{e}_{j,j}(\bar{a}x-x\bar{a})\hat{e}_{i,i},
$$
$$
\hat{e}_{i,i}\Delta(x)\hat{e}_{i,i}=\hat{e}_{i,i}(\bar{a}x-x\bar{a})\hat{e}_{i,i}
$$
are the same as the proofs of the appropriate equalities in the proof of Theorem \ref{3.11}.
If, for arbitrary elements $y$, $z$ in $M(\Omega,B_{sa}(H))$,
$$
\hat{e}_{k,k}y\hat{e}_{l,l}=\hat{e}_{k,k}z\hat{e}_{l,l}
$$
for each pair $k$, $l$ of natural numbers,
then $y=z$. Hence,
$$
\Delta(x)=\bar{a}x-x\bar{a}\in M(\Omega,B_{sa}(H)).
$$
But, $\Delta(x)\in \mathcal{J}$, since $x\in \mathcal{J}$. Therefore
$$
\Delta(x)=\bar{a}x-x\bar{a}\in \mathcal{J}
$$
for any element $x\in \mathcal{J}$. So, $\Delta$ is a derivation
on $\mathcal{J}$. The proof is complete.
\end{proof}

In particular, Theorem \ref{4.4} implies the following theorem.

\begin{theorem} \label{4.51}
Let $\Omega$ be an arbitrary set, $M(\Omega,B_{sa}(H))$ be the
Jordan algebra of all maps of $\Omega$ to $B_{sa}(H)$.
Then any 2-local inner derivation on the Jordan algebra $M(\Omega,B_{sa}(H))$ is a derivation.
\end{theorem}

Let $\Omega$ be a hyperstonean compact, $C(\Omega)$
denotes the algebra of all ${\mathbb C}$-valued (${\mathbb R}$-valued) continuous maps on $\Omega$.
There exists a subalgebra $\mathcal{N}$ in $M(\Omega,B_{sa}(H))$ which is a (real) von Neumann algebra
with the center isomorphic to $C(\Omega)$ (see \cite[Page 12]{AFN3}). More precisely
$\mathcal{N}$ is a (real) von Neumann algebra of type I.

The vector space
$$
\mathcal{N}_{sa}=\{a\in \mathcal{N}: a^*=a\}
$$
of all self-adjoint elements in $\mathcal{N}$
is a Jordan algebra with respect to Jordan multiplication
$$
a\cdot b=\frac{1}{2}(ab+ba), a,b\in \mathcal{N}_{sa}.
$$
The Jordan algebra $\mathcal{N}_{sa}$ is a Jordan subalgebra of $M(\Omega,B_{sa}(H))$
containing the element $x_o$ and the family $\{\hat{e}_{i,i}\}_{i=1}^\infty\cup\{s_{i,j}\}_{i,j=1,i\neq j}^\infty$.
Hence, by Theorem \ref{4.4}, we have the following statement.

\begin{theorem} \label{4.52}
Every 2-local spatial derivation
on $\mathcal{N}_{sa}$ implemented by elements in $M(\Omega,B_{sa}(H))$
is a derivation.
\end{theorem}

Note that, if ${\mathbb F}={\mathbb C}$, then Theorem \ref{4.52} follows by Theorem 3.4 in \cite{AKNA} and
the theorem in \cite{AA1}.

Let $\mathcal{K}(H)$ be the C$^*$-algebra of all compact
operators on the Hilbert space $H$, $\mathcal{K}_{sa}(H)$ be
the Jordan algebra of all self-adjoint compact operators on the Hilbert space $H$.
Let $\Omega$ be a topological space. Then
the vector space $C(\Omega,\mathcal{K}_{sa}(H))$ of all continuous maps
from $\Omega$ to $\mathcal{K}_{sa}(H)$ is a Jordan subalgebra of $M(\Omega,B_{sa}(H))$
containing the element $x_o$ and the family $\{\hat{e}_{i,i}\}_{i=1}^\infty\cup\{s_{i,j}\}_{i,j=1,i\neq j}^\infty$.
Therefore by Theorem \ref{4.4} we have the following theorem.

\begin{theorem} \label{4.6}
Every 2-local spatial derivation
on the Jordan algebra $C(\Omega,\mathcal{K}_{sa}(H))$ implemented by elements in $M(\Omega,B_{sa}(H))$
is a derivation.
\end{theorem}

Note that, if $\Omega$ is a compact Hausdorff space and ${\mathbb F}={\mathbb C}$, then Theorem \ref{4.6}
follows by Theorem 2.12 in \cite{JP17}. Also, if $\Omega$ is a locally compact Hausdorff space and ${\mathbb F}={\mathbb C}$, then Theorem \ref{4.6}
follows by Theorem 2.12 in \cite{JP17} and Theorem 2.8 in \cite{JP17b}.

\section{Local derivations on Jordan algebras of self-adjoint operator-valued maps}

This section is devoted to derivations and local derivations of Jordan algebras.

Let $\mathcal{A}$ be an algebra (nonassociative in general).
A linear map $\Delta : \mathcal{A}\to \mathcal{A}$ is called a local derivation, if for
any element $x\in \mathcal{A}$ there exists a derivation
$D_{x}:\mathcal{A}\to \mathcal{A}$ such that $\Delta (x)=D_{x}(x)$.

Let $\mathcal{A}$ be an associative algebra. A derivation $D$ on $\mathcal{A}$
is called an inner derivation, if there exists an element $a\in
\mathcal{A}$ such that
$$
D(x)=ax-xa, x\in \mathcal{A}.
$$
A linear map $\Delta : \mathcal{A}\to \mathcal{A}$ is called a local inner derivation,
if for any element $x\in \mathcal{A}$ there exists an element
$a\in \mathcal{A}$ such that $\Delta (x)=ax-xa$.

Let $\Delta$ be a local derivation on $\mathcal{J}$.
$\Delta$ is called a local inner derivation, if for each element
$x\in \mathcal{J}$ there is an inner derivation $D$ on $\mathcal{J}$ such that $\Delta(x)=D(x)$.

Let $\mathcal{A}$ be an associative unital algebra over an arbitrary field.
Then the vector space $\mathcal{A}$ with respect to the operation of Jordan multiplication
$$
a\cdot b=\frac{1}{2}(ab+ba), a, b\in \mathcal{A}
$$
is a Jordan algebra. This Jordan algebra will be denoted by $(\mathcal{A}, \cdot)$.

Let $\Delta$  be a local inner Jordan derivation of $(\mathcal{A},\cdot)$. Then for each element $x\in \mathcal{A}$
there is an inner derivation $D$ of $(\mathcal{A},\cdot)$ such that $\Delta(x)=D(x)$ and
$$
D=\sum_{k=1}^m D_{a_k,b_k}=D_{\frac{1}{4}(\sum_{k=1}^l [a_k,b_k])}
$$
for some elements
$a_1$, $a_2$, $\dots,$ $a_m$, $b_1$, $b_2$, $\dots,$ $b_m$ in $\mathcal{A}$.
Hence, $D$ is also an inner derivation of $\mathcal{A}$. So, $\Delta$ is a local inner derivation
of $\mathcal{A}$. Thus, every local inner derivation of the Jordan algebra $(\mathcal{A},\cdot)$
is a local inner derivation of the associative algebra $\mathcal{A}$.
Therefore, if every local inner derivation of the associative algebra $\mathcal{A}$ is a derivation,
then every local inner derivation of the Jordan algebra $(\mathcal{A},\cdot)$ is also a derivation.
If every Jordan derivation on an associative algebra $\mathcal{A}$ is an associative derivation, then
the converse of the last statement is also true, i.e., if every local inner derivation of
the Jordan algebra $(\mathcal{A},\cdot)$ is a derivation,
then every local inner derivation of the associative algebra $\mathcal{A}$ is also a derivation.

$H$ be an arbitrary Hilbert space over ${\mathbb F}={\mathbb R}$, ${\mathbb C}$, and, let $B(H)$ be the algebra of
all bounded linear operators on $H$. Let $\Omega$ be an arbitrary set, $M(\Omega,B_{sa}(H))$ be the
Jordan algebra of all maps from $\Omega$ to $B_{sa}(H)$.

In this section we describe local derivations on some class of subalgebras of $M(\Omega,B_{sa}(H))$.

Given $t\in \Omega$, $\phi_t : M(\Omega,B(H))\to B(H)$ will denote the
$*$-homomorphism defined by $\phi_t(x) = x(t)$, $x\in M(\Omega,B(H))$. The space $M(\Omega,B(H))$ is an $B(H)$-bimodule
with products $(ax)(t)=ax(t)$ and $(xa)(t)=x(t)a$, for every $a\in B(H)$,
$x\in M(\Omega,B(H))$. The map $\phi_t : M(\Omega,B(H))\to B(H)$ is an $B(H)$-module homomorphism.

Given an arbitrary set $\Omega$, the $*$-homomorphism which maps each element $a$ in $B(H)$ to the constant function
$\Omega\to \{a\}$ will be denoted by $\hat{a}$. The map
$$
\phi(a)=\hat{a}, a\in B(H),
$$
is an $B(H)$-module homomorphism of $B(H)$ to $M(\Omega,B(H))$.

\begin{definition}
Let $A$ be a Jordan algebra and, let $B$ be a Jordan subalgebra of $A$.
A local derivation $\nabla$ on $B$ is called local spatial derivation
with respect to the derivations implemented by elements in $A$, if
for every element $x\in B$ there exists a spatial derivation $D$
on $B$ implemented by elements in $A$ such that $\Delta(x)=D(x)$.
\end{definition}

Let $\mathcal{A}=B_{sa}(H)$, $\mathcal{K}_{sa}(H)$, and let $\Omega$ be a topological space.
Let $M(\Omega,\mathcal{A})$ be the Jordan algebra of all maps from $\Omega$ to $\mathcal{A}$, and let
$C(\Omega,\mathcal{A})$ be the Jordan algebra of all continuous maps from $\Omega$ to $\mathcal{A}$.
Then the following theorem takes place.

\begin{theorem} \label{5.1}
Let $\mathcal{J}$ be one of the Jordan algebras $M(\Omega,\mathcal{A})$, $C(\Omega,\mathcal{A})$
and let $\nabla$ be a local spatial derivation on $\mathcal{J}$ with respect to the derivations
implemented by elements in $M(\Omega,B_{sa}(H))$. Then $\nabla$ is a derivation.
\end{theorem}

\begin{proof}
We can prove that the map $\phi_t\nabla \phi:\mathcal{A}\to \mathcal{A}$ is a local inner derivation for every $t\in \Omega$.
Indeed, for any $a$, $b\in \mathcal{A}$, we have
\[
\phi_t\nabla \phi(a+b)=\phi_t\nabla(\hat{a}+\hat{b})
\]
\[
=\phi_t(\nabla(\hat{a})+\nabla(\hat{b}))=(\nabla(\hat{a})+\nabla(\hat{b}))(t)
\]
\[
=\nabla(\hat{a})(t)+\nabla(\hat{b})(t)=\phi_t(\nabla(\hat{a}))+\phi_t(\nabla(\hat{b}))
\]
\[
=\phi_t\nabla \phi(a)+\phi_t\nabla \phi(b).
\]
Similarly,
\[
\phi_t\nabla \phi(\lambda a)=\lambda\phi_t\nabla \phi(a).
\]
Hence, $\phi_t\nabla \phi$ is linear.
For each $a\in \mathcal{A}$ there exists a derivation $D_{\hat{a}}$ on $\mathcal{J}$ such that
$$
\nabla(\hat{a})=D_{\hat{a}}(\hat{a}).
$$
Note that $\phi_tD_{\hat{a}}\phi$ is a derivation on $\mathcal{A}$. Indeed,
it is clear that $\phi_tD_{\hat{a}}\phi$ is linear on $\mathcal{A}$.
For any $b$, $c\in \mathcal{A}$, we have
\[
\phi_tD_{\hat{a}}\phi(bc)=\phi_tD_{\hat{a}}(\phi(b)\phi(c))=
\]
\[
=\phi_t(D_{\hat{a}}(\phi(b))\phi(c)+\phi(b)D_{\hat{a}}(\phi(c)))
\]
\[
=D_{\hat{a}}(\phi(b))(t)\phi(c)(t)+\phi(b)(t)D_{\hat{a}}(\phi(c))(t)
\]
\[
=\phi_tD_{\hat{a}}\phi(b)c+b\phi_tD_{\hat{a}}\phi(c),
\]
where $\phi(b)(t)=b$, $\phi(c)(t)=c$.
Hence, $\phi_tD_{\hat{a}}\phi$ is a derivation on $\mathcal{A}$.

$\phi_tD_{\hat{a}}\phi$ can be weak-continuously extended to a derivation on $B_{sa}(H)$.
Every derivation on $B_{sa}(H)$ is inner by \cite{U}. So, $\phi_tD_{\hat{a}}\phi$ is a spatial
derivation implemented by elements in $B_{sa}(H)$.

Therefore, $\phi_t\nabla \phi$ is a local derivation and is a spatial derivation on $\mathcal{A}$ by theorem 5.4 in \cite{AKP}.

Let $a_t$ be an element in $B_{sa}(H)$ such that
$$
\phi_t\nabla \phi(x)=a_tx-xa_t, x\in B_{sa}(H).
$$
Let $\bar{a}(t)=a_t$, $t\in \Omega$. Then
$$
\nabla ({\bf x})(t)=(\bar{a}{\bf x}-{\bf x}\bar{a})(t), {\bf x}\in M(\Omega,B_{sa}(H)),
$$
i.e.
$$
\nabla ({\bf x})=\bar{a}{\bf x}-{\bf x}\bar{a}, {\bf x}\in M(\Omega,B_{sa}(H))
$$
and $\bar{a}\in M(\Omega,B_{sa}(H))$. Hence $\nabla$ is a derivation. The proof is compete.
\end{proof}

Note that, if the topological space $\Omega$ in Theorem \ref{5.1} is a locally compact Hausdorff space,
then Theorem \ref{5.1} follows by Theorem 5.4 in \cite{AKP}.

The authors thank the anonymous referee for the valuable suggestions
and comments.

\end{document}